\documentclass[12pt,a4paper]{article}

\usepackage[T1]{fontenc}
\usepackage[utf8]{inputenc}

\usepackage{textcomp}
\usepackage[official]{eurosym}
\usepackage{amssymb,mathrsfs,xparse}
\usepackage{amsthm}
\usepackage{mathtools}
\usepackage{latexsym}

\newtheorem{theorem}{Theorem}[section]

\newcommand{\kare}{$\square$}

\begin{document}
	\title{A solution of a problem about Erd\H{o}s space}
	\author{S\"uleyman \"Onal$^{(a)}$, Servet Soyarslan\footnote{Corresponding author. \newline E-mail addresses: osul@metu.edu.tr (S. Önal), servet.soyarslan@gmail.com (S. Soyarslan). } 
	}
	
	\maketitle	
	
	{\scriptsize a. Middle East Technical University, Department of Mathematics, 06531 Ankara, Turkey. } 
		
\begin{abstract}
 For Erd\H{o}s space, $\mathfrak{E}$, let us define a topology, $\tau_{clopen}$, which is generated  by  clopen subsets of $\mathfrak{E}$.  A. V. Arhangel’skii and J. Van Mill  asked whether  the topology $\tau_{clopen}$   is compatible with the group structure on $\mathfrak{E}$.  In this paper, we give a negative answer for this question.
\end{abstract}

	

\section{Introduction and Terminology}
We let $\mathbb{Q}$, $\mathbb{R}$ and $\mathbb{R}^+$ denote the sets of rational numbers, real numbers and positive real numbers, respectively.  $\mathbb{N}^+$ denotes the set of positive natural numbers, i.e., $\mathbb{N}^+=\{1,2,3,\ldots\}$.    
Let us consider the Banach space $\ell_2\subseteq \mathbb{R}^{\mathbb{N}^+}.$ This space consists of  all sequences $x=(x_1,x_2,x_3,\ldots)\in \mathbb{R}^{\mathbb{N}^+}$ such that the series  $\sum_{k=1}^{\infty}|x_k|^2$ is convergent.
 The topology on  $\ell_2$ is generated by the norm $||x||=\sqrt{\sum_{k=1}^{\infty}|x_k|^2}$. The Erd\H{o}s space $\mathfrak{E}$ is a  subspace of $\ell_2$ such that $\mathfrak{E}$ consists of the all sequences, the  all components of which are  rational,  i.e.,  $\mathfrak{E}=\mathbb{Q}^{\mathbb{N}^+}\cap \ell_2$. The topology $\tau$ on  $\mathfrak{E}$ is the subspace topology inherited  from  $\ell_2$. 

In this paper our main space is $(\mathfrak{E},\tau)$. What we mean with an open  ball $B(x,r)$ is the set $B(x,r)=\{y\in \mathfrak{E}: ||x-y||<r\}$ where $r>0$. If we say "$O$ is an open set", we mean that $O\subseteq \mathfrak{E}$ and $O\in \tau$.  Let $O$ be a subset of $\mathfrak{E}$. We denote the \textit{interior}  of $O$ by $int(O)$, i.e., $int(O)=\{x\in O:\exists V_x\in \tau(x\in V_x\subseteq O)\}$.

We need the following basic facts. These  can be found in any proper book.

\begin{theorem}\label{T:1} (\cite{arhangel2008topological} Theorem 1.3.12) Let $G$ be a topological group and $e$ the identity element of $G$. If $U$ is an open subset of $G$ and $e\in U$, there is an open subset $V$ of $G$ such that $e\in V$ and $V+V\subseteq U$.
\end{theorem}

\begin{theorem}\label{T:2}(\cite{engelking1978dimension} p. 17, \cite{arhangel2018some} p. 220)
The only bounded  clopen subset of  $(\mathfrak{E}, \tau)$ is the emptyset. 
\end{theorem}

\section{The  solution of the problem}
Let $(\mathfrak{E},\tau)$ be the topological space as in the above. Let $\mathcal{B}$ be the set of all clopen subsets of $\mathfrak{E}$, i.e., $\mathcal{B}=\{U\in \tau: \mathfrak{E}-U\in \tau\}$.  Take $\mathcal{B}$ as the base for a new topology $\tau_{clopen}$ on $\mathfrak{E}$.  Is the topology $\tau_{clopen}$ compatible with the group structure on $\mathfrak{E}$ ?  In \cite{arhangel2018some},  Question 8.9,  A. V. Arhangel’skii and J. Van Mill  asked this question. The following theorem states that the answer  of the question is negative.

\begin{theorem}
	The topology $\tau_{clopen}$ is not compatible with the group structure on $\mathfrak{E}$.
\end{theorem}
\begin{proof}
(Outline of this proof: First we define a clopen subset $A_{\alpha, \beta}$ and using this set we define a clopen subset $O$ of $\mathfrak{E}$ with $0\in O$. After that, we show that $V+V\nsubseteq O$ for any clopen subset $V$ of $\mathfrak{E}$ with  $0\in V$.)

Fix  any $\alpha\in \mathbb{R}^+$ and $\beta\in \mathbb{R}^+$ with the condition that $\alpha^2\notin \mathbb{Q}$ and $\beta\notin \mathbb{Q}$. 

Take any $x=(x_1,x_2,x_3,\ldots)\in \mathfrak{E}$. If $\{m\in \mathbb{N}^+:\sqrt{\sum_{k=1}^m|x_k|^2}>\alpha\}\neq \emptyset$, then let us  say that  $m_{x,\alpha}$ \textit{exists} and  define $m_{x,\alpha}=\min\{m\in \mathbb{N}^+:\sqrt{\sum_{k=1}^m|x_k|^2}>\alpha\}$. If $\{m\in \mathbb{N}^+:\sqrt{\sum_{k=1}^m|x_k|^2}>\alpha\}=\emptyset$, then let us  say that $m_{x,\alpha}$ \textit{does not exist}. 

Now, define $$\mathfrak{E}_\alpha=\{x\in \mathfrak{E}:m_{x,\alpha} \text{ does not exist}\}
$$
and define $$A_{\alpha,\beta}=\{x=(x_1,x_2,\ldots)\in \mathfrak{E}:m_{x,\alpha} \text{  exists and } |x_l|<\beta \text{ for all } l>m_{x,\alpha}  \}\cup\mathfrak{E}_\alpha.$$

\textbf{Claim 1.} 
The open ball $B(0,\alpha)$ is a subset of $A_{\alpha,\beta}$. (Here, $0=(0,0,0,\ldots)$ is the identity element of the topological group $(\mathfrak{E},\tau)$).

\textbf{Proof of  Claim 1.}
Take any $x=(x_1,x_2,x_3,\ldots)\in B(0,\alpha)$. Then, for all $m\in \mathbb{N}^+$,  $\sqrt{\sum_{k=1}^m|x_k|^2}\leq ||x||<\alpha$. Thus, $m_{x,\alpha}$ does not exist. So, $x\in \mathfrak{E}_\alpha\subseteq A_{\alpha,\beta}$. \kare 

\textbf{Claim 2.} 
 $A_{\alpha,\beta}$ is a  closed subset of $\mathfrak{E}$, i.e., $\mathfrak{E}-A_{\alpha,\beta}\in \tau$.

\textbf{Proof of  Claim 2.}

To see $\mathfrak{E}-A_{\alpha,\beta}$ is open, take and fix  any $z=(z_1,z_2,z_3,\ldots)\in \mathfrak{E}-O$. So, $m_{z,\alpha}$ exists,    $\alpha< 
\sqrt{\sum_{k=1}^{m_{z,\alpha}}|z_k|^2}$ and there exists an  $l_0>m_{z,\alpha}$ with $|z_{l_0}|\geq \beta$ because $z\notin A_{\alpha,\beta}$. Then, $|z_{l_0}|> \beta$, because $z_{l_0}\in \mathbb{Q}$ and  $\beta \notin \mathbb{Q}$. Thus, we can define $r_0=\min \{|z_{l_0}|-\beta,\sqrt{
\sum_{k=1}^{m_{z,\alpha}}|z_k|^2} - \alpha\}$  and take the open ball $B(z,r_0)$. Take any $y=(y_1,y_2,y_3,\ldots)\in B(z,r_0)$.  

 $\sqrt{
\sum_{k=1}^{m_{z,\alpha}}|z_k|^2}-\sqrt{
\sum_{k=1}^{m_{z,\alpha}}|y_k|^2}\leq \sqrt{
\sum_{k=1}^{m_{z,\alpha}}|z_k-y_k|^2}\leq \sqrt{
\sum_{k=1}^{\infty}|z_k-y_k|^2}=||z-y||<r\leq \sqrt{
\sum_{k=1}^{m_{z,\alpha}}|z_k|^2} - \alpha$. So, $\sqrt{
	\sum_{k=1}^{m_{z,\alpha}}|z_k|^2}-\sqrt{
	\sum_{k=1}^{m_{z,\alpha}}|y_k|^2}<\sqrt{
	\sum_{k=1}^{m_{z,\alpha}}|z_k|^2} -\alpha$. Thus, $\alpha<
\sum_{k=1}^{m_{z,\alpha}}|y_k|^2$. 

So, $m_{y,\alpha}$ exists and because of the definition of $m_{y,\alpha}$, $m_{y,\alpha}\leq m_{z,\alpha}$. 

Thus, $l_0\geq m_{z,\alpha}\geq m_{y,\alpha}$ and $|z_{l_0}|-|y_{l_0}|\leq |z_{l_0}-y_{l_0}|=\sqrt{|z_{l_0}-y_{l_0}|^2}\leq \sqrt{
	\sum_{k=1}^{\infty}|z_k-y_k|^2}=||z-y||<r\leq|z_{l_0}|-\beta$. So, $|z_{l_0}|-|y_{l_0}|<|z_{l_0}|-\beta$. Thus, $\beta< |y_{l_0}|$.

Because ${m_{y,\alpha}}$ exists and there exists  $l_0>m_{y,\alpha}$ such that $\beta< |y_{l_0}|$, $y\notin A_{\alpha,\beta}$. Therefore $z\in B(z,r_0)\subseteq \mathfrak{E}-A_{\alpha,\beta}$. Hence, $A_{\alpha,\beta}$ is a  closed subset of $\mathfrak{E}$. \kare

\textbf{Claim 3.} 
$A_{\alpha,\beta}$ is an  open subset of $\mathfrak{E}$, i.e., $A_{\alpha,\beta}\in \tau$.

\textbf{Proof of  Claim 3.} Take and fix any $x=(x_1,x_2,x_3,\ldots)\in A_{\alpha,\beta}$. 
 There are two cases: $m_{x,\alpha}$ does not exist or $m_{x,\alpha}$ exists.

\textbf{Case 1:} $m_{x,\alpha}$ does not exist.

Then, $\{m\in \mathbb{N}^+:\sqrt{\sum_{k=1}^m|x_k|^2}>\alpha\}= \emptyset$. So, $\sqrt{\sum_{k=1}^m|x_k|^2}\leq\alpha$ for all $m\in \mathbb{N}^+$.
For this case either $||x||=\sqrt{\sum_{k=1}^{\infty}|x_k|^2}<\alpha$ or $||x||=\alpha$. 

If $||x||<\alpha $, then from Claim 1, $x\in B(0,\alpha)\subseteq A_{\alpha, \beta}$.

Now, suppose $||x||=\alpha$. Say $r_1=\sqrt{\alpha^2+\beta^2}-\alpha$. Then, to see the open ball $B(x,r_1)$ is a subset of  $A_{\alpha, \beta}$, take any $y=(y_1,y_2,y_3,\ldots)\in B(x,r_1)$. If $m_{y,\alpha}$ does not exist, then $y\in \mathfrak{E}_\alpha\subseteq A_{\alpha, \beta}$. If $m_{y,\alpha}$  exists, then aiming for a contradiction, suppose there exists an  $l>m_{y,\alpha}$ and $|y_l|\geq\beta$.  Then, $||y||\leq ||y-x||+||x||<r_1+||x||<\sqrt{\alpha^2+\beta^2}-\alpha+\alpha=\sqrt{\alpha^2+\beta^2}$. Therefore, $||y||<\sqrt{\alpha^2+\beta^2}$. So, $\alpha^2+\beta^2>||y||^2=\sum_{k=1}^{\infty}|y_k|^2\geq |y_l|^2+\sum_{k=1}^{m_{y,\alpha}}|y_k|^2>\beta^2+\alpha^2$. Thus, we get the contradiction $\alpha^2+\beta^2>\beta^2+\alpha^2$. From the contradiction, we get $|y_l|\leq \beta$. Becuse $|y_l|\in \mathbb{Q}$, $\beta\notin \mathbb{Q}$ and $l>{m_{y,\alpha}}$ is arbitrary, $|y_l|< \beta$ for all $l>{m_{y,\alpha}}$. Hence, $y\in A_{\alpha, \beta}$. Thus, $x\in int(A_{\alpha, \beta})$.  

\textbf{Case 2:} $m_{x,\alpha}$ exists.

Because $||x||=\sqrt{
\sum_{k=1}^{\infty}|x_k|^2}<\infty$, for $\varepsilon=\frac{\beta}{2}$ 
there exists an $l_0\in \mathbb{N}^+$ such that $\sqrt{\sum_{k=l_0}^{\infty}|x_k|^2}<\varepsilon=\frac{\beta}{2}$.  Thus, \begin{equation}\label{1}
	\sqrt{\sum_{k=l_0}^{\infty}|x_k|^2}<\frac{\beta}{2}.
\end{equation} 

\textbf{Subclaim 1.}
For any $y=(y_1,y_2,\ldots)\in \mathfrak{E}$ if $||y-x||<\frac{\beta}{2}$, then $|y_l|<\beta$ for all $l\geq l_0$. 

\newpage\textbf{Proof of Subclaim 1.}

   $-\sqrt{\sum_{k=l_0}^{\infty}|x_k|^2}+\sqrt{\sum_{k=l_0}^{\infty}|y_k|^2} \leq \sqrt{\sum_{k=l_0}^{\infty}|x_k-y_k|^2}\leq \sqrt{\sum_{k=1}^{\infty}|x_k-y_k|^2}$ $=   ||x-y||<\frac{\beta}{2}$. 
So, 
$-
\sqrt{\sum_{k=l_0}^{\infty}|x_k|^2}+\sqrt{\sum_{k=l_0}^{\infty}|y_k|^2}<\frac{\beta}{2}$.  Thus, from (\ref{1}), $\sqrt{\sum_{k=l_0}^{\infty}|y_k|^2}<\frac{\beta}{2}+\sqrt{\sum_{k=l_0}^{\infty}|x_k|^2}=\beta$. 
 Therefore,  $|y_l|\leq \sqrt{\sum_{k=l_0}^{\infty}|y_k|^2}<\beta$ where $l\geq l_0$. So, $|y_l|<\beta$ for all  $l\geq l_0$. The proof of Subclaim 1 is completed.

Now, say $$h_x=\min\{\big|\beta-|x_i|\big|:i\leq l_0\}$$ and 
	$$a_x= \begin{cases} 
	{\alpha}-
	\sqrt{\sum_{k=1}^{m_{x,\alpha}-1}|x_k|^2} & :m_{x,\alpha} >1\\
	$$1$$ & :m_{x,\alpha} =1  
\end{cases}
$$
(We don't need to case $a_x=1$. To get a well defined $r$, we write it. Note that  if $m_{x,\alpha}>1$, then ${\alpha}-
\sqrt{\sum_{k=1}^{m_{x,\alpha}-1}|x_k|^2}>0$, because $\sqrt{\sum_{k=1}^{m_{x,\alpha}-1}|x_k|^2}\leq \alpha$, so, $\sum_{k=1}^{m_{x,\alpha}-1}|x_k|^2\leq \alpha^2$. Because $\alpha^2\notin \mathbb{Q}$ and $\sum_{k=1}^{m_{x,\alpha}-1}|x_k|^2\notin \mathbb{Q}$, $\sum_{k=1}^{m_{x,\alpha}-1}|x_k|^2< \alpha^2$. Thus, $
\sqrt{\sum_{k=1}^{m_{x,\alpha}-1}|x_k|^2}<{\alpha}$.)

Now,  define $$r=\min\{\sqrt{\sum_{k=1}^{m_{x,\alpha}}|x_k|^2}-\alpha, \frac{\beta}{2}, h_x, a_x\}.$$ To see the open ball $B(x,r)$ is a subset of $A_{\alpha,\beta}$, take 
any  $y=(y_1,y_2,y_3,\ldots)\in B(x,r)$.

 $\sqrt{\sum_{k=1}^{m_{x,\alpha}}|x_k|^2}-\sqrt{\sum_{k=1}^{m_{x,\alpha}}|y_k|^2}\leq \sqrt{\sum_{k=1}^{m_{x,\alpha}}|x_k-y_k|^2}\leq \sqrt{\sum_{k=1}^{\infty}|x_k-y_k|^2}$ $=||x-y||<r\leq \sqrt{\sum_{k=1}^{m_{x,\alpha}}|x_k|^2}-\alpha$.  Thus, $\sqrt{\sum_{k=1}^{m_{x,\alpha}}|x_k|^2}-\sqrt{\sum_{k=1}^{m_{x,\alpha}}|y_k|^2}< \sqrt{\sum_{k=1}^{m_{x,\alpha}}|x_k|^2}-\alpha$ and so, $\sqrt{\sum_{k=1}^{m_{x,\alpha}}|y_k|^2} >\alpha$. Therefore, $m_{y,\alpha}$ exists and $m_{y,\alpha}\leq m_{x,\alpha}$.  

Now, we will see that $m_{y,\alpha}= m_{x,\alpha}$. If $m_{x,\alpha}=1$, then $m_{y,\alpha}= m_{x,\alpha}$ because $m_{y,\alpha}\leq m_{x,\alpha}$. If $m_{x,\alpha}>1$,  take any $n<m_{x,\alpha}$. Then,  $-\sqrt{\sum_{k=1}^{n}|x_k|^2}+\sqrt{\sum_{k=1}^{n}|y_k|^2}\leq \sqrt{\sum_{k=1}^{n}|x_k-y_k|^2}\leq \sqrt{\sum_{k=1}^{\infty}|x_k-y_k|^2}=||x-y||<r\leq a_x\leq \alpha -
\sqrt{\sum_{k=1}^{m_{x,\alpha}-1}|x_k|^2}$. 

So,  
$-\sqrt{\sum_{k=1}^{n}|x_k|^2}+\sqrt{\sum_{k=1}^{n}|y_k|^2}<\alpha-\sqrt{\sum_{k=1}^{m_{x,\alpha}-1}|x_k|^2}$. Therefore $\sqrt{\sum_{k=1}^{n}|y_k|^2}<\alpha+\sqrt{\sum_{k=1}^{n}|x_k|^2}-\sqrt{\sum_{k=1}^{m_{x,\alpha}-1}|x_k|^2}$. Because $n\leq m_{x,\alpha}-1$, $\sqrt{\sum_{k=1}^{n}|x_k|^2}-\sqrt{\sum_{k=1}^{m_{x,\alpha}-1}|x_k|^2}\leq 0$. 
Thus, $\sqrt{\sum_{k=1}^{n}|y_k|^2}<\alpha<\sqrt{\sum_{k=1}^{m_{y,\alpha}}|y_k|^2}$. So, $n<m_{y,\alpha}$. 
 Hence, $n<m_{y,\alpha}$ for all $n<m_{x,\alpha}$. Therefore, $m_{y,\alpha}\geq m_{x,\alpha}$. So, $m_{y,\alpha}= m_{x,\alpha}$.  

Fix any arbitrary $l> m_{y,\alpha}= m_{x,\alpha}$. We will see $|y_l|<\beta$. If  $l\geq l_0$, then because $ ||x-y||<r\leq \frac{\beta}{2}$, from Subclaim 1, $|y_l|<\beta$. If $l_0\geq l> m_{y,\alpha}= m_{x,\alpha}$ where $l_0> m_{y,\alpha}$, then  $-|x_l|+|y_l|\leq|x_l-y_l|\leq  ||x-y||<r\leq \big|\beta-|x_l|\big|$. On the other hand, $|x_l|<\beta$ because $x\in A_{\alpha, \beta}$, $m_{x,\alpha}$ exists and  $l>m_{x,\alpha}$. Thus, $-|x_l|+|y_l|<\big|\beta-|x_l|\big|=\beta -|x_l|$. Hence, $|y_l|<\beta$.

For Case 2, because $m_{y,\alpha}$ exists and $|y_l|<\beta$ for all $l>m_{y,\alpha}$, $y\in A_{\alpha, \beta}$. Therefore,  $B(x,r)\subseteq  A_{\alpha, \beta}$ and so, $x\in int(A_{\alpha, \beta})$.

Therefore,   $A_{\alpha,\beta}$ is an  open subset of $\mathfrak{E}$. \kare

Let $(\alpha_n)$ be a strictly increasing sequence of positive real numbers ($\alpha_1<\alpha_2<\alpha_3<\ldots$) such that $(\alpha_n)\rightarrow \infty$  and $\alpha_n^2\notin \mathbb{Q}$ for all $n\in \mathbb{N}^+$. And let $(\beta_n)$ be a strictly decreasing sequence of positive real numbers such that $(\beta_n)\rightarrow 0$  and $\beta_n\notin \mathbb{Q}$ for all $n\in \mathbb{N}^+$. Then, we can give the following claim.

\textbf{Claim 4.} Let
$O=\bigcap_{n\in \mathbb{N}^+}A_{\alpha_n,\beta_n}$. Then, the identity element $0\in O$ and $O$ is a  clopen subset of $\mathfrak{E}$, i.e., $O\in \tau$ and $\mathfrak{E}-O\in \tau$.

\textbf{Proof of  Claim 4.} From Claim 1, $0\in A_{\alpha_n,\beta_n}$ for all $n\in \mathbb{N}^+$. Thus,  $0\in O$.

From Claim 2, each $A_{\alpha_n,\beta_n}$ is closed. Thus,  $O=\bigcap_{n\in \mathbb{N}^+}A_{\alpha_n,\beta_n}$ is a  closed in $(\mathfrak{E},\tau)$.

To see $O$ is open in $(\mathfrak{E},\tau)$, take and fix any $x\in O$. Because $(\alpha_n)\rightarrow \infty$, there is an $n_0\in \mathbb{N}^+$ such that $||x||<\alpha_{n_0}$. Thus, $x\in B(0,\alpha_{n_0})$. Define $W=(\bigcap_{n\leq n_0}A_{\alpha_n,\beta_n})\cap(B(0,\alpha_{n_0}))$.  So, clearly $W$ is open, and $x\in W$  because $x\in B(0,\alpha_{n_0})$ and $x\in \bigcap_{n\in \mathbb{N}^+}A_{\alpha_n,\beta_n}$. Now, fix any $m\in \mathbb{N}^+$. We will see that $W\subseteq A_{\alpha_m,\beta_m}$.  If $m\leq n_0$, then $W\subseteq \bigcap_{n\leq n_0}A_{\alpha_n,\beta_n}\subseteq A_{\alpha_m,\beta_m}$. Suppose $m\geq n_0$. Because $(\alpha_{n})$ is strictly increasing sequence, $\alpha_{m}\geq \alpha_{n_0}$. Then,  $W\subseteq B(0,\alpha_{n_0})\subseteq B(0,\alpha_{m})$. Thus,  $W\subseteq B(0,\alpha_{m}) \subseteq A_{\alpha_{m},\beta_m}$
 because  we know that $B(0,\alpha_{m}) \subseteq A_{\alpha_{m},\beta_m}$ from Claim 1. Thus, $W\subseteq O=\bigcap_{n\in \mathbb{N}^+}A_{\alpha_n,\beta_n}$ because $m$ is an arbitrary element of $\mathbb{N}^+$. Hence, $x\in W\subseteq O$ and $W$ is open. Therefore,   $O$ is an  open subset of $\mathfrak{E}$. \kare

\textbf{Claim 5.} If $V$ is any open unbounded subset of $(\mathfrak{E},\tau)$ such that $0=(0,0,0,\ldots)\in V$, then $V+V\nsubseteq O$.

\textbf{Proof of Claim 5.}
 Fix any open unbounded subset  $V$ of $\mathfrak{E}$ such that the identity element   $0\in V$. Then, there exists an $r^*>0$ such that the open ball $B(0,r^*)$ is a subset of  $V$. We can find an $n^*\in \mathbb{N}^+$ such that $0<\frac{1}{n^*}<r^*$. Because $(\alpha_{n})\rightarrow \infty$ and $(\beta_n)\rightarrow 0$, there exist $m_1,m_2\in \mathbb{N}^+$ such that $0<\beta_{m_1}<\frac{1}{n^*}$ and $n^*<\alpha_{m_2}$.  Say $m^*=\max\{m_1,m_2\}$. Then, $\beta_{m^*}<\frac{1}{n^*}$ and $n^*<\alpha_{m^*}$.
  Because $V$ is unbounded,  there exists an $x\in V$ such that $\alpha_{m^*}<||x||$.  Thus, there exists an $m\in \mathbb{N}^+$ such that $\alpha_{m^*}<\sqrt{\sum_{k=1}^{m} |x_k|^2}$. Therefore, $m_{x,{\alpha_{m^*}}}$  exists. Now, fix any $l^*>m_{x,{\alpha_{m^*}}}$ and  a rational number $q$ such that $\frac{1}{\beta_{m^*}}<q<r^*$. Let $e^{l^*}=(e_1,e_2,\ldots)\in \mathfrak{E}$ such that $e_{l^*}=1$ and $e_k=0$ where $k\neq {l^*}$.
 
 \textbf{Case 1:} $x_{l^*}\geq 0$.
 
  Say $y=q.e^{l^*}$. So, $y_{l^*}=q$ and $y_k=0$ for all $k\neq {l^*}$ where $y=(y_1,y_2,y_3,\ldots)$. Then, $y\in B(0,r^*)$ because $||y-0||=||y||=q<r^*$. So, $y\in V$. Now, define $z=x+y$. Thus, $z\in V+V$. Also $\sqrt{\sum_{k=1}^{m_{x,{\alpha_{m^*}}}} |z_k|^2}=\sqrt{\sum_{k=1}^{{m_{x,{\alpha_{m^*}}}}} |x_k|^2}>\alpha_{m^*}$. So, $m_{z,{\alpha_{m^*}}}$ exists and from definition of $m_{z,{\alpha_{m^*}}}$, $m_{z,{\alpha_{m^*}}}\leq m_{x,{\alpha_{m^*}}}\leq l^*$.  
  Thus, $|z_{l^*}|=|x_{l^*}+q|=x_{l^*}+q> \frac{1}{\beta_{m^*}}$. So, $z\notin A_{\alpha_{m^*},\beta_{m^*}}$. Therefore, $z\notin O$. Hence  $z\in V+V$ and $z\notin O$. 

\textbf{Case 2:} $x_{l^*}<0$.

 Say $y=-q.e^{l^*}$. Define $z=x+y$, in a similar manner to Case 1, $z\in V+V$ and $z\notin O$. 
 
 Therefore $V+V\nsubseteq O$. \kare
 
 From Claim 4 and Claim 5, there is a clopen subset $O$ of $(\mathfrak{E},\tau)$  with $0\in O$ such that $V+V\nsubseteq O$ if  $V$ is any open unbounded subset  of $(\mathfrak{E},\tau)$ with $0\in V$. From Theorem \ref{T:2}, if $V$ is any clopen  subset  of $(\mathfrak{E},\tau)$ with $0\in V$, then $V$ is unbounded.  Therefore,  $O\in \tau_{clopen}$ with $0\in O$ such that  $V+V\nsubseteq O$ for any $V\in \tau_{clopen}$ with $0\in V$.  Hence, from Theorem \ref{T:1}, the topology $\tau_{clopen}$ is not compatible with the group structure on $\mathfrak{E}$.
 
\end{proof}
 In Claim 5 which is in the proof above, for the clopen set $O$, actually we showed  that if $K$ is unbounded subset of $\mathfrak{E}$,  then $K+B(0,\varepsilon)\nsubseteq O$
 for any $\varepsilon>0$.
\bibliographystyle{plain}

\end{document}